\documentclass[a4paper,12pt]{article}
\usepackage{amssymb,amsmath,amsthm}

\newcommand\cA{{\mathcal A}}

\newcommand\cF{{\mathcal F}}

\newcommand\cI{{\mathcal I}}

\newcommand\cP{{\mathcal P}}

\newcommand\cR{{\mathcal R}}

\newcommand\cT{{\mathcal T}}

\theoremstyle{plain}
\newtheorem{theorem}{Theorem}[section]
\newtheorem{lemma}[theorem]{Lemma}
\newtheorem{corollary}[theorem]{Corollary}

\theoremstyle{definition}

\newtheorem{proposition}[theorem]{Proposition}

\newtheorem{observation}[theorem]{Observation}

\newcommand\lref[1]{Lemma~\ref{lem:#1}}
\newcommand\tref[1]{Theorem~\ref{thm:#1}}
\newcommand\cref[1]{Corollary~\ref{cor:#1}}

\newcommand\sref[1]{Section~\ref{sec:#1}}
\newcommand\pref[1]{Proposition~\ref{prop:#1}}

\newcommand\oref[1]{Observation~\ref{obs:#1}}

\begin{document}
\author{D\'aniel Gerbner, Gyula O.H. Katona, D\"om\"ot\"or P\'alv\"olgyi and Bal\'azs Patk\'os}

\title{Majority and Plurality Problems}
\maketitle

\begin{abstract}
Given a set of $n$ balls each colored with a color, a ball is said to be majority, $k$-majority, plurality if
its color class has size larger than half of the number of balls, has size at least $k$, has size larger than any other color class; respectively. We address the problem of finding the minimum number of queries (a comparison of a pair of balls if they have the same color or not) that is needed to decide whether a majority, $k$-majority or plurality ball exists and if so then show one such ball. We consider both adaptive and non-adaptive strategies and in certain cases, we also address weighted versions of the problems.\end{abstract}


\section{Introduction}

Two very much investigated problems in combinatorial search theory 
are the so-called majority and
plurality problems. In this context, we are given $n$ balls in an urn, 
each colored with one color. A majority
ball is one such that its color class has size strictly larger
than $n/2$. A plurality ball is one such that its color class is
strictly larger than any other color class. The aim is either to
decide whether there exists a majority/plurality ball or even to
show one (if there exists one). Note that if the number of colors
is two, then the majority and the plurality problems coincide.
Although there are other models (e.g \cite{DKW}), in the original
settings a query is a pair of balls and the answer to the query
tells us whether the two balls have the same color or not.
Throughout the paper we consider queries of this sort.

The first results in this area of combinatorial search theory are
due to Fisher and Salzberg \cite{FS} and Saks and Werman
\cite{SW}. In \cite{FS} it is proved that if the number of
possible colors is unknown, then the minimum number of queries in
an adaptive search for a majority ball is $\lfloor 3n/2 \rfloor
-2$, while \cite{SW} contains the result that if the number of
colors is two, then the minimum number of queries needed to find a
majority ball is $n-b(n)$, where $b(n)$ is the number of 1's in
the binary representation of $n$. The latter result was later
reproved in a simpler way by Alonso, Reingold, and Schott
\cite{ARS} and Wiener \cite{W}.

The adaptive version of the plurality problem was first considered
by Aigner, De Marco, and Montangero in \cite{ADM}, where they
showed that for any fixed positive integer $c$, if the number of
possible colors is at most $c$, then the minimum number of queries
needed in an adaptive search for a plurality ball is of linear
order, and the constants depend on $c$. Non-adaptive and other
versions of the plurality problem were considered in \cite{A}.

Non-adaptive strategies were also studied by Chung, Graham, Mao
and Yao \cite{CGMY1, CGMY2}. They showed a linear upper bound for
the majority problem in case the existence of a majority color is
assumed. They mention a quadratic lower bound without this extra
assumption. We precisely determine the minimum number of queries needed. 
They also obtain lower and upper bounds on the
plurality problem in the non-adaptive case. We improve those
bounds and find the correct asymptotics of the minimum number of
queries.

\vskip 0.3truecm

To state our results we introduce some notations. $M_c(n)$ denotes
the minimum number of queries that is needed to determine if there
exists a majority color and if so, then to show one ball of that
color and $P_c(n)$ denotes the minimum number of queries that is
needed to determine if there exists a plurality color and if so,
then to show one ball of that color. In both cases the subscript
$c$ stands for the number of possible colors. The corresponding
non-adaptive parameters are denoted by $M^*_c(n)$ and $P^*_c(n)$.
A ball is said to be $k$-majority if its color class contains at
least $k$ balls. $M_c(n,k)$ denotes the minimum number of queries
that is needed to determine if there exists a $k$-majority color
and if so, then to show one ball of that color and $M^*_c(n,k)$
denotes the parameter of the non-adaptive variant.

We also consider weighted problems. Let $S=\{w(1), \dots, w(n)\}$
be a multiset of positive numbers, where $w(i)$ is considered to
be the weight of the $i$th ball. The total weight $w=w(S)$ of the
balls is $\sum_{i=1}^nw(i)$. The weight $w(T)$ of a subset $T
\subseteq [n]$ is $\sum_{i \in T}w(i)$. A color is majority if its
color class $C$ satisfies $w(C)>w/2$ and $k$-majority if $w(C) \ge
k$ holds. A color is said to be plurality if the weight $w(C)$ of
its color class $C$ is strictly greater than the weights of all
the other color classes. The appropriate parameters are denoted by
$M_c(S), M_c(S,k), M_c^*(S), M_c^*(S,k)$ and $P_c(S), P_c^*(S)$.

For a set $Q$ of queries we define the \textit{query graph} $G_Q$
to be the graph where the vertices correspond to balls and two
vertices are joined by an edge if and only if there exists a query
in $Q$ that asks for the comparison of the two corresponding
balls.

Throughout the paper $\log$ stands for the logarithm of base 2.

\vskip 0.3truecm

In \sref{tools} we describe a result of Hayek, Kutin and Melkebeek
\cite{HKM} and some further observations that will serve as tools
in our proofs. In \sref{majo} we characterize the query graphs
that solve the non-adaptive $k$-majority problem. As a corollary
we obtain the following theorem.

\begin{theorem}
\label{thm:kmajo} Suppose $n \ge c >2$. Then $ M^*_c(n)=
\lceil\lceil n/2 \rceil n/2 \rceil$.
\end{theorem}

In the rest of \sref{majo} we consider the weighted adaptive
majority problem and obtain lower bounds on $M_2(S)$ and
$M_2(n,k)$. Our bounds are always at least as good as the one
established by Aigner \cite{A}, and in some cases our bounds are
better.

\sref{plur} contains bounds on $P^*_c(n)$ and $P^*(S)$. Our main
result concerning the plurality problem is the following theorem.

\begin{theorem}
\label{thm:nonadaptplur}
For any pair of integers $n$ and $c$, the following holds:
\[
\left\lceil\frac{1}{2} \left(n-1-\frac{n-1}{c-1}\right)n\right\rceil \le
P^*_c(n) \le \frac{c-2}{2(c-1)}n^2+n.
\]
\end{theorem}

\section{Tools}
\label{sec:tools}

In this section we introduce a result of Hayes, Kutin and
Melkebeek \cite{HKM} that we will use in \sref{majo}. We also make
an easy observation that will serve as a tool in \sref{plur}.

Let us start with describing the context of the result by Hayes,
Kutin and Melkebeek. Let $f$ be a Boolean function of $n$
variables $x_1,x_2,...,x_n$. A \textit{parity question} is a
subset $T \subseteq [n]$ and the answer to this question is
$\sum_{i\in T}x_i$ modulo 2. Let us define
$\mathcal{D}^{\cP\cA\cR\cI\cT\mathcal{Y}}(f)$ to be the minimum
number of parity questions needed to determine the value of $f$.

\begin{lemma}[Hayes, Kutin and Melkebeek, Lemma 17 in \cite{HKM}]
\label{lem:also} Let $f$ be a Boolean function on $\{0,1\}^n$. If
$\mathcal{D}^{\cP\cA\cR\cI\cT\mathcal{Y}}(f)\le d$, then $2^{n-d}$
divides $|f^{-1}(1)|$.
\end{lemma}

The proof of the lemma is a straightforward extension of a lower
bound by Rivest and Vuillemin \cite{RV} for standard decision
trees. Note that if $|T|=2$, then the answer to a parity question
tells us whether the variables in $T$ have the same value or not.
Let $f$ be the function that takes value $1$ if there is majority
among the values of the variables and 0 if there is no majority.
It is easy to see that $f$ is a Boolean function.
Using \lref{also} Hayes, Kutin and Melkebeek gave a surprisingly short
proof of the lower bound of the result of Saks and Werman \cite{SW}
which states that for two colors the minimum number of queries needed to find a
majority ball is $n-b(n)$,
where $b(n)$ is the number of 1's in the binary representation of $n$.
Note also that \lref{also} cannot be applied if $n$ is
odd. On the other hand it yields a stronger result than that of
Saks and Werman in case $n$ is even: \lref{also} gives the same
lower bound even if the aim is only to decide whether there is
majority and one is allowed to ask parity questions where $|T|=2$
is not required.
The proof is to simply apply \lref{also} with $f$ being the
majority function and notice that the largest two power that divides
$\sum_{i>n/2} {n \choose i}=2^{n-1}-{n\choose n/2}/2$ is $2^{b(n)}$
(this follows e.g. from Kummer's theorem).

\vskip 0.5truecm

Let us finish this section with an easy observation on how
complete multipartite graphs can be used in detecting color
classes of balls. The next observation will be only used in the
non-adaptive case.

\begin{observation}
\label{obs:multipart} If $G_Q$ is a complete multipartite graph
and $u,v$ are vertices of different classes corresponding to balls
of the same color $C$, then all the balls of color $C$ can be
identified.
\end{observation}

\begin{proof}
Any vertex $x$ corresponding to a ball of color $C$ is joined by an edge in $G_Q$ to either $u$ or $v$ and therefore there is a query in $Q$ asking whether
the ball corresponding to $x$ has the same color as the balls corresponding to $u$ and $v$.
\end{proof}

\section{Majority}
\label{sec:majo} In this section we consider majority problems.
Let us start with the non-adaptive case. If the number of possible
colors is two, then Aigner solved the $k$-majority problem
provided $n/2<k$.

\begin{theorem}[Aigner, \cite{A}]
\label{thm:aig} For $n \ge 3$ \[
 M^*_2(n,k)=\left\{\begin{array}{ll}
 n-1 & \textrm{if $n<2k-1$ }
 \\  n-2 & \textrm{if $n=2k-1$}. \end{array} \right.\]
\end{theorem}

Let us continue with the more general, weighted $k$-majority
model. The next theorem characterizes the query graphs that solve
the weighted $k$-majority problem provided some extra assumptions
are satisfied. For simplicity, we will assume that the vertex set
of the query graph is $[n]$. Given a multiset
$S=\{w_1,w_2,...,w_n\}$ of weights let $\mathcal{F}=\{F \subset
[n]: w(F)\ge k\}$ be the family of the $k$-majority sets. Let
$\cF_0$ denote the subfamily of minimal sets in $\cF$.

\begin{theorem}
\label{thm:charact} Suppose there are no 1-element sets in
$\mathcal{F}$. Then

\noindent (i) If each member of $\cF_0$ induces a connected
subgraph of the query graph $G_Q$, then $G_Q$ solves the weighted
$k$-majority problem.

\noindent (ii) If $2w([n]) < (k+1)(c+1)-2$, $c>2$ and $G_Q$ solves the
weighted $k$-majority problem, then each member of $\mathcal{F}_0$
induces a connected subgraph of $G_Q$.

\noindent (iii) Considering the non-weighted version, suppose $k$
is an integer. If $n\le ck-k-c+2$, $c>2$ and $G_Q$ solves the
$k$-majority problem, then each member of $\mathcal{F}$, i.e. any
set with at least $k$ elements, induces a connected subgraph of
$G_Q$.

\end{theorem}

\begin{proof} (i) If the induced subgraph on $F$ is connected, then one can easily check if $F$ is monochromatic. By the assumption, a minimal $k$-majority set of each candidate for $k$-majority can be checked.

\bigskip

(ii) Suppose $F_0\in \mathcal{F}_0$ is not connected. Suppose the
answers are according to a coloring satisfying the following: one
component of $F_0$ is blue, all the others are red, and the
remaining $c-2$ colors are used on the other vertices in such a
way that none of those colors are in $k$-majority. If such a
coloring exists, there is no $k$-majority, but it cannot be
distinguished from the case when every vertex in $F_0$ is blue
(what would make blue a majority color).

To prove that such a coloring exists, we have to show that there is a partition of
$[n] \setminus F_0$ into $c-2$ color classes $A_1,A_2,...,A_{c-2}$,
such that $w(A_i)\le k-1$ for all $i$. The conditions
imply $w([n] \setminus F_0)< (k+1)(c-1)/2$. Notice that this is a Bin Packing problem;
we have to prove that certain items with weight sum $<(k+1)(c-1)/2$ fit into $c-2$ bins.
Indeed, it is well-known (and easy to see) that if $c-1(\ge 2)$ bins are required then the sum has to be at least $(c-2)\lceil(k+1)/2\rceil+\lfloor(k+1)/2\rfloor\ge (k+1)(c-1)/2$, we are done.


\bigskip

(iii) The proof goes similarly to the previous case. Suppose $F
\in \cF$ is not connected. By removing elements from $F$ one by
one such that we make sure that at least two components of $F$ do
not get totally removed, we obtain a subset $F'$ of $F$ such that
$F' \in \cF_0$ and $F'$ is not connected. By the assumption on
$n$, we know that $|[n] \setminus F'|\le (c-2)(k-1)$. Let us
partition $[n] \setminus F'$ into $c-2$ color classes each
containing at most $k-1$ elements. If one component of $F_0$ is
colored blue and all other components are colored red, then there
is no $k$-majority. But this coloring cannot be distinguished from
the case when every vertex in $F_0$ is blue (what would make blue
a majority color).
\end{proof}

\begin{corollary}\label{cor:nonadkmaj}
Suppose $c>2$, $n>k>n/2$ and $n>1$. Then a query graph $G_Q$
solves the $k$-majority problem if and only if $G_Q$ is
$(n-k+1)$-connected.
\end{corollary}

\begin{proof}
Suppose first that $G_Q$ is $(n-k+1)$-connected and let $F \in
\cF$, i.e. $|F| \ge k$ and thus $|[n] \setminus F| \le n-k$. By
definition $G_Q$ stays connected after removing the vertices in
$[n] \setminus F$ and thus by \tref{charact} (i) $G_Q$ solves
$k$-majority.

Suppose now $G_Q$ solves the $k$-majority problem. A simple
calculation shows that the assumption of \tref{charact} (iii)
holds, thus all $k$-subsets of $[n]$ induce a connected subgraph
of $G_Q$, thus $G_Q$ is $(n-k+1)$-connected.
\end{proof}

\begin{proof}[Proof of \tref{kmajo}]
Clearly, the majority problem is the $k$-majority problem with
$k=\lfloor n/2\rfloor +1$. Thus by \cref{nonadkmaj} any query
graph $G_Q$ that solves the majority problem must be $\lceil n/2
\rceil$-connected  and it is well known that the minimum number of
edges that $\lceil n/2 \rceil$-connected graphs can have is
$\lceil \lceil n/2 \rceil n/2\rceil$.


\end{proof}

Let us note here that the upper bound of \tref{kmajo} holds also
in the weighted case, but such general lower bound cannot be found
 without extra assumptions on the multiset $S$ of weights. Indeed,
 if $w_1>\sum_{i=2}^nw_i$, holds, then without any query one knows that the ball with
 weight $w_1$ is a majority ball.

\vskip 0.6truecm

Let us now turn our attention to adaptive majority problems. We
will only address problems where the number of colors is two.
Apart from the results by Fisher and Salzberg \cite{FS} and Saks
and Werman \cite{SW} mentioned already in the Introduction, we are
aware of one more major result. If $\mu(n)$ denotes the largest
integer $l$ such that $2^l$ divides $n$, then Aigner's result can
be formulated in the following way:

\begin{theorem}
\label{thm:aig2} For any pair of integers $n \ge k>n/2$, the
inequality $M_2(n,k) \ge n-1-\mu({n-1 \choose k-1})$ holds.
\end{theorem}

Note that this result is a generalization of the theorem of Saks
and Werman as if $n$ is even, then for $k= n/2+1$ we have
$\mu(\binom{n-1}{n/2})=b(n)-1$, where $b(n)$ is the number of 1's in the binary representation of $n$.

\vskip 0.3truecm

We establish two lower bounds on $M_2(n,k)$, which easily follow
from \lref{also}.

\begin{proposition}
\label{prop:p1} Let $k>n/2$. Then $$M_2(n,k) \ge
n-1-\mu\left(\sum_{i=k}^n{n \choose i}\right).$$
\end{proposition}

\begin{proof}
Let $f:\{0,1\}^n \rightarrow \{0,1\}$ be the function defined by
$f({\bf x})=1$ if and only if $\sum_{i=1}^nx_i\ge k$ or
$\sum_{i=1}^nx_i\le n-k$, i.e., there is a $k$-majority color.
This is clearly a Boolean function and
$|f^{-1}(1)|=2\sum_{i=k}^n{n \choose i}$. The statement now
follows from \lref{also}.
\end{proof}

We can compare this bound to Aigner's lower bound. The example in
\cite{A} which shows that Aigner's bound is not optimal is $n=9$
and $k=6$, where his result only gives $5$, while the truth is
$7$. Our bound establishes the correct value $7$. On the other
hand if $n=10$ and $k=8$, our bound yields only $6$, Aigner's
result gives $7$, while the truth is $8$.

\begin{proposition}
\label{prop:p2} Let $k>n/2$. Then

(i) Let us fix an arbitrary ball. The number Fix$_2(n,k)$ of
questions needed to determine if the fixed ball is a $k$-majority
ball is at least $ n-1-\mu(\sum_{i=k}^n{n-1 \choose i-1}).$

(ii) $M_2(n,k) \ge n-2-\mu(\sum_{i=k}^n{n-1 \choose i-1}).$

\end{proposition}

\begin{proof}  \lref{also} implies (i) as before, and (ii)  follows
as Fix$_2(n,k) \le M_2(n,k)+1$. Indeed, if we solve the
$k$-majority problem and the answer is that there is no
$k$-majority color, then the fixed ball cannot be a $k$-majority
ball. If the answer is a $k$-majority ball, then after comparing
this to our fixed ball we can decide whether the fixed ball is
also $k$-majority or not.
\end{proof}

Comparing \tref{aig2}, \pref{p1} and \pref{p2} (ii), one observes
that any two statements imply the third one. First note that
$n-2-\mu(\sum_{i=k}^n{n-1 \choose i-1})$ can be written as
$n-1-\mu(2\sum_{i=k}^n{n-1 \choose i-1})$ and also the identity
$2\sum_{i=k}^n{n-1 \choose i-1}={n-1 \choose k-1}+\sum_{i=k}^n{n
\choose i}$ holds. Finally note that for any integers $\alpha,
\beta, \gamma$ with $\alpha+\beta=\gamma$ there is no unique
minimum of $\mu(\alpha)$, $\mu(\beta)$ and $\mu(\gamma)$. This
means that we have also obtained a new proof of \tref{aig2}.

\vskip 0.5truecm

Let us now consider the weighted (adaptive) majority problem with
two colors. Suppose there are $p\neq 0$ ways to partition the
multiset $S$ into two parts of equal weight. Then

\begin{proposition}
\label{prop:p3} (i) At least $n-1-\mu(p)$ questions are needed.

(ii) In case $p$ is even, $n-2$ questions are enough
\end{proposition}

\begin{proof}
Let $f:\{0,1\}^n \rightarrow \{0,1\}$ be the function defined by
$f({\bf x})=1$ if and only if $\sum_{i: x_i=0} w(i)=\sum_{i:
x_i=1} w(i)$. Again, (i) follows from \lref{also}.

For (ii), note that $n$ has to be at least $3$ (since zero weights
are not allowed). Now we claim that from any three elements there
are two that are in different equipartitions an even number of
times (and hence also in the same partition an even number of
times, so comparing them will keep the property that $p$ is even).
The proof is the following. For each partition add an edge between
the elements that are in different parts. Each partition gives
zero or two edges between our three elements. Thus two of them
will be connected with an even number of edges, we are done.
\end{proof}

Note that this means $M_2(S)=n-1$ iff $\mu(p)=0$. If $\mu(p)=1$,
then $M_2(S)= n-2$, but the opposite direction is not true, as
shown by the following example: $S=\{1, 10, 11, 100, 101, 110,
111\}$. Here $M_2(S)=5$ even though $p=4$.

\vskip 1truecm

Another possible assumption about the multiset $S$ of weights is
that ``every element matters'', i.e. for every $s\in S$ there
exists a coloring of $S\setminus \{s\}$ with two colors (red and
blue) such that the majority color is different if we extend this
coloring by giving $s$ color red or blue. We say that a multiset
$S$ of weights is \textit{non-slavery} if the above condition is
satisfied.

\begin{proposition}
\label{prop:p4} For every non-slavery multiset $S$ of weights the
inequality $M_2(S)\ge \lfloor n/2 \rfloor$ holds.
\end{proposition}

\begin{proof} Let $s\in S$ be the element with the smallest
weight. There is a partition $S\setminus \{s\}= A \cup B$ such
that both $w(A \cup \{s\})\ge w(V)/2$ and $w(B \cup \{s\})\ge
w(V)/2$. Suppose $|A| \ge |B|$. The Adversary can declare all balls
corresponding to weights in $B$ to be colored blue and apply the following
strategy: for any query $(a,b)$ with $a \in A \cup \{s\}, b \in B$ the answer is
DIFFERENT COLOR and for any other query the answer is SAME COLOR. As $s$ is the smallest
weight, then by the non-slavery property of $S$ we know that the answer to
the majority problem is different if all balls corresponding to weights in $A \cup \{s\}$ are colored red
or if there is at least one ball colored blue among those balls. But no matter
which set $Q$ of $l$ queries are made with $l < \lfloor n/2 \rfloor$, there will
be at least one component $C$ in $G_Q$ that lies totally within $A \cup \{s\}$.
Now the following two colorings are admissible to $Q$ and give different
answers to the majority problem:
\begin{enumerate}
\item
all balls corresponding to weights in $A\cup \{s\}$ are colored red
\item
balls corresponding to weights in $C$ are colored blue and to those in $(A \cup \{s\}) \setminus C$
are colored red.
\end{enumerate}
\end{proof}

\section{Plurality}
\label{sec:plur}

As we pointed out in the Introduction, the problem of finding a
plurality ball is the same as the problem of finding a majority
ball if the number of possible colors is two. In this section we
consider the case $c \ge 3$ and determine $P^*_c(n)$
asymptotically. Our main goal is to prove \tref{nonadaptplur} and
then to obtain an upper bound on $P^*_c(S)$ in the weighted case
for general weight set $S$.

The lower bound of \tref{nonadaptplur} immediately follows from the following lemma.

\begin{lemma}
\label{lem:mindeg} If $Q$ is a set of queries that solve the
problem, then the minimum degree in $G_Q$ is larger than
$n-1-\lceil\frac{n-1}{c-1}\rceil$. Furthermore, if $n-1 \equiv 1$
mod $c-1$, then the minimum degree in $G_Q$ is larger than
$n-1-\lfloor\frac{n-1}{c-1}\rfloor$.
\end{lemma}

\begin{proof}[Proof of Lemma]
We consider the two cases separately. First, suppose that $n-1
\not\equiv 1$ mod $c-1$ and the degree of a vertex $x$ is at most
$n-1-\lceil\frac{n-1}{c-1}\rceil$. Then one can partition $V(G_Q)
\setminus \{x\}$ into $c-1$ sets whose size differ by at most 1
(we call such a partition equipartition), and $V \setminus (N(x)
\cup \{x\})$ contains $T_1$, one of the largest sets. If the
Adversary answers the queries in the following way, we will not be
able to tell whether there is a plurality color or not: for any
$u,v \ne x$ the answer is SAME COLOR if and only if $u,v$ belong
to the same set of the partition and the answer is DIFFERENT COLOR
for any query $(u,x)$. In this way, if $x$ has the same color as
$T_1$, a subset of its non-neighbors, then this color is
plurality, while if $x$ has color different from any other vertex,
then there is no plurality.

Let us consider the case $n-1 \equiv 1$ mod $c-1$. Then in a
$(c-1)$-equipartition of $V(G_Q) \setminus \{x\}$ there is one set
$T_1$ which is one larger than all other sets. Then consider such
a partition of $V(G_Q) \setminus \{x\}$ where $V \setminus (N(x)
\cup \{x\})$ contains a small set $T$ of the partition. Let the
Adversary answer the queries as above: for any $u,v \ne x$ the
answer is SAME COLOR if and only if $u,v$ belong to the same set
of the partition and the answer is DIFFERENT COLOR for any query
$(u,x)$. Then we cannot answer the problem even if we have the
extra information that balls belonging to the same partite set are
colored with the same color. Indeed, if $x$ has the same color as
all balls in $T$, then there is no plurality: $T_1$ and $T \cup
\{x\}$ have the same size. Or if $x$ is colored with a color that
is different from the color of all other balls, then any ball from
$T_1$ is a plurality ball.
\end{proof}

\begin{proof}[Proof of the upper bound in \tref{nonadaptplur}] It is enough to show a graph $G_{c,n}$ with
$\frac{c-2}{2(c-1)}n^2+n$ edges such that no matter what colors
the balls have, we are able to solve the problem after receiving
the answers to queries corresponding to edges of $G_{c,n}$. Let
$G_{c,n}$ be the $(c-1)$-partite Tur\'an graph on $n$ vertices
with a spanning cycle added to each partite set $V_1, \dots,
V_{c-1}$.

\vskip 0.3truecm

First, let us observe that if for some $x \in V_i, y \in V_j$ with
$i\ne j$ the answer to the query $(x,y)$ is SAME COLOR,
then after receiving answers to all queries we are able to
determine the whole color class of $x$ and $y$. Indeed, as $x$ and
$y$ belong to different partite sets, for any $v \in V$ at least
one of the queries $(x,V),(y,v)$ is asked and therefore we receive
a SAME COLOR answer if $v$ belongs to the color class of $x$ and
$y$ and a DIFFERENT COLOR answer is $v$ is of different color.

Let $k$ denote the number of color classes $C_1,...,C_k$ that
intersect at least two of the $V_i$'s and let $l$ denote the
number of partite sets that are not contained in $\cup_{i=1}^k
C_i$. Any of the remaining $c-k$ color classes is contained in one
of the partite sets not covered by $\cup_{i=1}^k C_i$, thus by the
pigeonhole principle $l \le c-k$. Moreover, if $l=c-k$, then to
cover all partite sets, all $c-k$ colors need to be used in
different partite sets and thus they should be of the form $V_j
\setminus \cup_{i=1}^k C_i$. As we are able to determine all color
classes, the proof is finished in this case.

From now on we assume $l \le c-k-1$. Let us suppose first that
$k=0$. Then as all color classes are included in one of the
partite sets we have that all but at most one partite sets form
one color class each and the last partite set is the union of at
most two color classes which we can identify due to the additional
spanning cycle.

Thus we may suppose that $l \le c-k-1$ and $k \ge 1$ hold. This
means that the number of already covered partite sets is $c-1-l
\ge k$, thus at least one of $C_1,...,C_k$ has size at least
$\lfloor \frac{n}{c-1}\rfloor$. Therefore it is enough to prove
that we are able to identify all other color classes that have
size at least $\lfloor \frac{n}{c-1}\rfloor$. As all remaining
color classes are contained in one of the $V_i$'s, their size is
at most $\lceil \frac{n}{c-1}\rceil$. Thus if a $V_j$ contains at
least two points from $\cup_{i=1}^k C_i$, then it cannot contain a
color class of size at least $\lfloor \frac{n}{c-1}\rfloor$. On
the other hand if a $V_j$ contains at most one point of
$\cup_{i=1}^k C_i$, then, due to the spanning cycle, we are able
to tell whether it contains a color class of size $\lceil
\frac{n}{c-1}\rceil$ or $\lfloor \frac{n}{c-1}\rfloor$.
\end{proof}

One can improve the upper bound of \tref{nonadaptplur} for $c=3$. It is not hard to show that the following constructions for $G_Q$ correspond to query sets that solve the problem: if $n=2k$, then consider the complete bipartite graph $K_{k,k}$ on the partite sets $\{u_1,u_2,...,u_k\}$, $\{v_1,v_2,...,v_k\}$. Add the edges of the paths $(u_1u_2...u_k)$ and $(v_1v_2..v_k)$ and remove the edges of the matching $\{(u_i,v_i): i=2,...,k-1\}$. While if $n=2k+1$, then consider the complete bipartite graph $K_{k+1,k}$ on the partite sets $\{u_1,u_2,...,u_k,u_{k+1}\},\{v_1,v_2,...,v_k\}$. Add the edges of the path $(u_1u_2...u_ku_{k+1})$. These observations yield to the following theorem.

\begin{theorem}
\label{thm:3} (i) $P^*_3(2k)=k(k+1)$,

\noindent (ii) $\frac{1}{2}(k+1)(2k+1) \le P^*_3(2k+1) \le
\frac{1}{2}(k+1)(2k+1)+k-1$.

\end{theorem}

To obtain the lower bound in (i), one has to apply the $n-1 \equiv 1$ mod $c-1$ case of \lref{mindeg}. Details are left to the reader.

\vskip 0.5truecm

Finally, we turn our attention to the non-adaptive weighted plurality problem, i.e. determining $P^*_c(S)$ for a multiset $S$ of weights. \tref{nonadaptplur} shows that in general we cannot hope for anything better than the number of edges of the balanced complete $(c-1)$-partite graph on $n$ vertices. Our last theorem states that for any multiset of weights the number of edges of the balanced complete $c$-partite graph on $n$ vertices and a linear number of additional queries can solve the problem.

\begin{theorem}
\label{thm:nonadaptplurweight} For any multiset $S$ of $n$ weights the inequality $P^*_c(S) \le \frac{c-1}{2c}n^2+n-c$ holds.
\end{theorem}

\begin{proof} Let us partition $S$ into $c$ classes $A_1,\dots, A_c$ with
$w(A_1)\ge w(A_2)\ge \dots \ge w(A_c)$ such that $w(A_i)-w(v) \le
w(A_c)$ for any $i\le c$ and any $v \in A_i$. It is easy to see
that such a partition exists. Indeed, starting from any partition
if $v$ and $A_i$ violates this property, we can remove $v$ from
$A_i$ and add it to $A_c$. Let $B$ denote the set of all pairs
$(v,A_i)$ such that $v \in A_i$ and $w(A_i)-w(v)>w(A_c)$. Put
$O=O(A_1,A_2,...,A_c)=\sum_{(v,A_i) \in B}w(A_i)-w(v)-w(A_c)$.
Note that after removing $v$ from $A_i$ and adding it to $A_c$ the
value of $O$ strictly decreases, hence this process must stop
after a certain time and the resulting partition satisfies the
required property. Note that $w(A_i)-w(v) \le w(A_c) \le w(S)/c$,
and the sum of the weights in a plurality color must be larger
than $w(S)/c$, the average weight of color classes. It implies
that a plurality color (in fact any color with weight larger than
$w(S)/c$) either contains vertices from more than one class, or is
equal to $A_i$ for some $i$.

Let $G_0$ be the complete $c$-partite graph with parts $A_1,\dots, A_c$, and add a spanning tree into every part. Let $G$ be the resulting graph. It is well-known that $G_0$ cannot have more edges than the $c$-partite Tur\'an graph on $n$ vertices, and the spanning trees add $n-c$ additional edges. We claim that this graph solves the weighted plurality problem. If there is a color with weight larger than $w(S)/c$, then it is either a class $A_i$, and in that case the spanning tree shows that it is monocolored, or it contains vertices from at least two parts, and then by \oref{multipart} we can identify all its vertices. Hence all the colors with weight larger than $w(S)/c$ are completely identified, all one has to do is to check if the largest weight appears only once or more.
\end{proof}

\section*{Acknowledgement.} We would like to thank Bal\'azs Keszegh and G\'abor Wiener for fruitful discussions.

\end{document}